\documentclass{amsart}
\usepackage{amsmath}
\usepackage{amscd}
\usepackage{amssymb}
\usepackage{amsthm}
\RequirePackage{filecontents}
\usepackage{fullpage}
\usepackage[numbers]{natbib}  
\usepackage[draft]{hyperref}

\theoremstyle{plain}
\newtheorem{prop}{Proposition}
\newtheorem{lem}[prop]{Lemma}

\theoremstyle{definition}
\newtheorem{defn}[prop]{Definition}
\newtheorem{rem}[prop]{Remark}
\newtheorem{ex}[prop]{Example}

\title{Poincar\'e square series of small weight}
\author{Brandon Williams }

\address{Department of Mathematics \\ University of California \\ Berkeley, CA 94720}

\email{btw@math.berkeley.edu}

\begin{document}

\nocite{*}

\maketitle

\begin{abstract} We extend the author's earlier computation and give coefficient formulas for the (quasimodular) Poincar\'e square series of weight $3/2$ and weight $2$ for the dual Weil representation for an even lattice.
\end{abstract}

\section{Introduction}

Let $\rho^*$ denote the dual Weil representation for an even lattice $\Lambda$ of signature $(b^+,b^-)$ with quadratic form $q$; that is, the unitary representation of the metaplectic group $\tilde \Gamma = Mp_2(\mathbb{Z})$ given on the standard generators $S,T$ by $$\rho^*(S) \mathfrak{e}_{\gamma} = \frac{\mathbf{e}((b^+ - b^-) / 8)}{\sqrt{|\Lambda'/\Lambda|}} \sum_{\beta \in \Lambda'/\Lambda} \mathbf{e}\Big( \langle \gamma, \beta \rangle \Big) \mathfrak{e}_{\beta}$$ and $$\rho^*(T) \mathfrak{e}_{\gamma} = \mathbf{e}\Big( - q(\gamma) \Big) \mathfrak{e}_{\gamma},$$ where $\Lambda'/\Lambda$ is the discriminant group of $\Lambda$ and $\mathfrak{e}_{\gamma}$, $\gamma \in \Lambda'/\Lambda$ is the natural basis of the group ring $\mathbb{C}[\Lambda'/\Lambda]$. \\

In \cite{W1} the author gave formulas for a family of modular forms $Q_{k,m,\beta} \in M_k(\rho^*)$ with rational coefficients with the property that $Q_{k,m,\beta} - E_k$ span all cusp forms (where $E_k$ is the Eisenstein series) and which are characterized through the Petersson scalar product by $$(f, Q_{k,m,\beta} - E_k) = 2 \frac{\Gamma(k-1)}{(4m\pi)^{k-1}} \sum_{\lambda = 1}^{\infty} \frac{c(m \lambda^2, \lambda \beta)}{\lambda^{2k-2}}, \; \; f \in S_k(\rho^*).$$ This construction is valid in weights $k \ge 5/2$; and for weight $k \ge 3$, $Q_{k,m,\beta}$ is just the zero-value of a Jacobi Eisenstein series for an appropriate generalization of $\rho^*$. \\

The purpose of this note is to extend this construction to weights $k = 3/2$ and $k = 2.$ These cases are more difficult because the Eisenstein series may fail to define a modular form; in fact, it is not hard to find lattices where $M_k(\rho^*) = 0$. For example, the space of scalar-valued modular forms of weight $2$ is zero. These weights remain relevant to the problem that motivated \cite{W1}; namely, computing spaces of obstructions for the existence of Borcherds products. Modular forms of weight $k = 3/2$ resp$.$ $k=2$ are obstructions to the existence of Borcherds products on Grassmannians $G(2,1)$ (equivalently, scalar modular forms) resp$.$ $G(2,2)$ (equivalently, Hilbert modular forms) as explained in \cite{Bo}. \\

The failure of the Jacobi Eisenstein series of weight $k \le 5/2$ to define a Jacobi form is closely related to the failure of the usual Eisenstein series of weight $k - 1/2$ to define a modular form. In particular, $k = 2$ is the most difficult weight to study, since Eisenstein series of weight $3/2$ tend to be mock modular forms that require a complicated real-analytic correction term to transform correctly under $\tilde \Gamma$. \\

Even in the cases where there are no cusp forms, the computation of $Q_{k,m,\beta}$ may be interesting. We give two examples at the end where the equality of quasimodular forms $Q_{2,1,0} = E_2$ results in a nontrivial identity.

\textbf{Acknowledgements:} I am grateful to Richard Borcherds for helpful discussions.

\tableofcontents

\section{Notation and background}
We abbreviate $e^{2\pi i x}$ by $\mathbf{e}(x).$ \\

$\Lambda$ is an even lattice with nondegenerate quadratic form $q$ of signature $(b^+,b^-)$ and dimension $e = b^+ + b^-$. The corresponding discriminant form is $\Lambda'/\Lambda$ with group ring $\mathbb{C}[\Lambda'/\Lambda]$. The natural basis of $\mathbb{C}[\Lambda'/\Lambda]$ is denoted $\mathfrak{e}_{\gamma}$, $\gamma \in \Lambda'/\Lambda,$ and $\langle -,- \rangle$ is the scalar product $$\Big \langle \sum_{\gamma \in \Lambda'/\Lambda} a_{\gamma} \mathfrak{e}_{\gamma}, \sum_{\gamma \in \Lambda'/\Lambda} b_{\gamma} \mathfrak{e}_{\gamma} \Big \rangle = \sum_{\gamma \in \Lambda'/\Lambda} a_{\gamma} \overline{b_{\gamma}}.$$ We write $d_{\gamma}$ and $d_{\beta}$ to denote the denominator, or level, of $\gamma,\beta \in \Lambda'/\Lambda$; these are the smallest natural numbers such that $$d_{\gamma} \cdot \gamma, \; \; d_{\beta} \cdot \beta \in \Lambda.$$

 We denote by $\tilde \Gamma = Mp_2(\mathbb{Z})$ the metaplectic group, which is the double cover of $SL_2(\mathbb{Z})$ consisting of pairs $(M,\phi)$, with $M = \begin{pmatrix} a & b \\ c & d \end{pmatrix} \in SL_2(\mathbb{Z})$ and a branch $\phi$ of the square root of $c \tau + d$ on the upper half-plane $\mathbb{H}$. The square root $\phi$ is almost always omitted from notation. Recall that $\tilde \Gamma$ is generated by the elements $S = \Big( \begin{pmatrix} 0 & -1 \\ 1 & 0 \end{pmatrix}, \sqrt{\tau} \Big)$ and $T = \Big( \begin{pmatrix} 1 & 1 \\ 0 & 1 \end{pmatrix}, 1 \Big)$ with defining relations $S^8 = I$ and $S^2 = (ST)^3$. \\
 
 The Weil representation is the map $$\rho : \tilde \Gamma \longrightarrow \mathrm{Aut} \, \mathbb{C}[\Lambda'/\Lambda],$$ $$\rho(T)\mathfrak{e}_{\gamma} = \mathbf{e}\Big( q(\gamma) \Big) \mathfrak{e}_{\gamma}, \; \; \rho(S) \mathfrak{e}_{\gamma} = \frac{\mathbf{e}( (b^- - b^+)/8)}{\sqrt{|\Lambda'/\Lambda|}} \sum_{\beta \in \Lambda'/\Lambda} \mathbf{e}\Big( \langle \gamma, \beta \rangle \Big) \mathfrak{e}_{\beta},$$  with unitary dual $\rho^*$. \\
 
  $\mathcal{H}$ denotes the Heisenberg group. The underlying set is $\mathcal{H} = \mathbb{Z}^3$, and the group operation is $$(\lambda_1,\mu_1,t_1) \cdot (\lambda_2,\mu_2,t_2) = (\lambda_1 + \lambda_2, \mu_1 + \mu_2, t_1 + t_2 + \lambda_1 \mu_2 - \lambda_2 \mu_1).$$ For any $\beta \in \Lambda'/\Lambda$, we define a finite analogue of the classical Schr\"odinger representation by $$\sigma_{\beta} : \mathcal{H} \longrightarrow \mathrm{Aut}\, \mathbb{C}[\Lambda'/\Lambda], \; \; \sigma_{\beta}(\lambda,\mu,t) \mathfrak{e}_{\gamma} = \mathbf{e}\Big( \mu \langle \beta, \gamma \rangle + (t - \lambda \mu) q(\beta) \Big) \mathfrak{e}_{\gamma - \lambda \beta}.$$ The Jacobi group is the semidirect product $\mathcal{J} = \mathcal{H} \rtimes \tilde \Gamma$ by the action $$(\lambda,\mu,t) \cdot \begin{pmatrix} a & b \\ c & d \end{pmatrix} = (a \lambda + c\mu, b\lambda + d\mu,t),$$ and for $\beta \in \Lambda'/\Lambda$ we denote by $\rho_{\beta}$ the representation $$\rho_{\beta} : \mathcal{J} \longrightarrow \mathrm{Aut} \, \mathbb{C}[\Lambda'/\Lambda]$$ that restricts to $\rho$ on $\tilde \Gamma$ and to $\sigma_{\beta}$ on $\mathcal{H}$. (See section 3 of \cite{W1}.) \\
 
 The weight-$k$ action of $\tilde \Gamma$ on holomorphic functions $$f : \mathbb{H} = \{\tau = x + iy: \, y > 0\} \longrightarrow \mathbb{C}$$ is written using the Petersson slash operator: $$f \Big|_{k,\rho^*} M (\tau) = (c \tau + d)^{-k} \rho^*(M)^{-1} f\Big( \frac{a \tau + b}{c \tau + d} \Big), \; \; M = \begin{pmatrix} a & b \\ c & d \end{pmatrix} \in \tilde \Gamma,$$ and a modular form of weight $k$ for $\rho^*$ is a holomorphic function $f$ satisfying $f|_{k,\rho^*} M = f$ for all $M \in \tilde \Gamma$ and the usual growth condition at $\infty$. We denote by $M_k(\rho^*)$ the space of weight-$k$ modular forms and by $S_k(\rho^*)$ the space of weight-$k$ cusp forms, which are modular forms that vanish in $\infty$. \\
 
  Both $M_k(\rho^*)$ and $S_k(\rho^*)$ are always finite-dimensional. In weight $k > 2$ the dimensions of $M_k(\rho^*)$ and $S_k(\rho^*)$ can be computed effeciently through the Riemann-Roch theorem. This tends to fail in weight $k \le 2$, where most formulas instead produce expressions for the ``Euler characteristic" $\mathrm{dim}\, M_k(\rho^*) - \mathrm{dim}\, S_{2-k}(\rho)$ instead. On the other hand, Ehlen and Skoruppa \cite{ES} have described an algorithm that computes dimensions in weight $k=2$ and $k=3/2$ that in practice seems quite efficient, relying on the known structure for $M_0(\rho^*)$ (which consists of constant Weil invariants) and $M_{1/2}(\rho^*)$ (where the components are theta series and related oldforms by the Serre-Stark theorem \cite{SS}) which was computed more precisely in \cite{Sk}). \\
 
Similarly, for $\beta \in \Lambda'/\Lambda$ and $m \in \mathbb{Z} - q(\beta),$ the weight-$k$ and index-$m$ action of the Jacobi group $\mathcal{J}$ on holomorphic functions $f(\tau,z)$ (where $\tau \in \mathbb{H}$ and $z \in \mathbb{C}$) is \begin{align*} &\quad f \Big|_{k,m,\rho^*} (\zeta,M) (\tau,z) \\ &= (c \tau + d)^{-k} \mathbf{e}\Big( m \lambda^2 \tau + 2m \lambda z + m(\lambda \mu + t) - \frac{cm (z + \lambda \tau + \mu)^2}{c \tau + d} \Big) \rho_{\beta}^*(\zeta,M)^{-1} f\Big( \frac{a \tau + b}{c \tau + d}, \frac{z + \lambda \tau + \mu}{c \tau + d} \Big),\end{align*} for $M = \begin{pmatrix} a & b \\ c & d \end{pmatrix} \in \tilde \Gamma$ and $\zeta = (\lambda,\mu,t) \in \mathcal{H}$, and we define Jacobi forms of weight $k$ and index $m$ for $\rho_{\beta}^*$ to be holomorphic functions $f$ satisfying $f |_{k,m,\rho^*_{\beta}} (\zeta,M) = f$ for all $(\zeta,M) \in \mathcal{J}$ and satisfying a growth condition at $\infty$: the Fourier coefficient of $q^n \zeta^r \mathfrak{e}_{\gamma}$ must be $0$ unless $r^2 \le 4mn$. Here, $q = \mathbf{e}(\tau)$ and $\zeta = \mathbf{e}(z).$ \\

In weight $k \ge 3$, the basic example of a Jacobi form is the Jacobi Eisenstein series $$E_{k,m,\beta}(\tau,z) = \sum_{(\zeta,M) \in \mathcal{J}_{\infty} \backslash \mathcal{J}} \mathfrak{e}_0 \Big|_{k,m,\rho_{\beta}^*} (\zeta,M).$$ (We interpret $\mathfrak{e}_0$ as a constant function of $\tau$ and $z$.) Here, $\mathcal{J}_{\infty}$ is the subgroup of $\mathcal{J}$ that fixes $\mathfrak{e}_0$: it is generated by $S^2$ and $T$ and by the elements of $\mathcal{H}$ of the form $(0,\mu,t)$, \, $\mu,t \in \mathbb{Z}$. A formula for the Fourier coefficients of $E_{k,m,\beta}(\tau,z)$ is given in \cite{W1}. The expressions for the Fourier coefficients there make sense for $k \in \{3/2,2,5/2\}$ as well, but the result is usually not a Jacobi form; we also denote this series by $E_{k,m,\beta}$. \\

At several points throughout this note we will consider the integral $$I(k,y,\omega,s) = y^{1-k-2s} e^{2\pi \omega y} \int_{-\infty}^{\infty} (t+i)^{-k} (t^2 + 1)^{-s} \mathbf{e}(-\omega y t) \, \mathrm{d}t$$ and a Dirichlet series $\tilde L(n,r,\gamma,s)$ which are defined in section 3.

\section{The real-analytic Jacobi Eisenstein series}

Fix an even lattice $\Lambda$, an element $\beta \in \Lambda'/\Lambda$ and a positive number $m \in \mathbb{Z} - q(\beta).$

\begin{defn} The \textbf{real-analytic Jacobi Eisenstein series} of weight $k$ and index $m$ twisted at $\beta$ is $$E_{k,m,\beta}^*(\tau,z,s) = \frac{1}{2} \sum_{c,d}\sum_{\lambda \in \mathbb{Z}} (c \tau + d)^{-k} |c \tau + d|^{-2s} \mathbf{e}\Big( m \lambda^2 (M \cdot \tau) + \frac{2m \lambda z - cm z^2}{c \tau + d} \Big) \rho^*(M)^{-1} \mathfrak{e}_{\lambda \beta}.$$ Here, $c,d$ runs through all pairs of coprime integers, and $M = \begin{pmatrix} a & b \\ c & d \end{pmatrix} \in \tilde \Gamma$ is any element with bottom row $(c,d)$. This series converges locally uniformly for $\mathrm{Re}[s] > \frac{3-k}{2}.$
\end{defn}

\begin{rem} Writing $$E_{k,m,\beta}^*(\tau,z,s) = y^{-s} \cdot \sum_{(\zeta,M) \in \mathcal{J}_{\infty} \backslash \mathcal{J}} (y^s \mathfrak{e}_0)\Big|_{k,m,\rho_{\beta}^*} (\zeta,M)$$ makes it clear that $E_{k,m,\beta}^*(\tau,z,s)$ transforms under the Jacobi group by $$E_{k,m,\beta}^*\Big( \frac{a \tau + b}{c \tau + d}, \frac{z}{c \tau + d}, s \big) = (c \tau + d)^k |c \tau + d|^{2s} \mathbf{e}\Big( \frac{mcz^2}{c \tau + d} \Big) \rho^*(M) E_{k,m,\beta}^*(\tau,z,s)$$ and $$E_{k,m,\beta}^*(\tau,z+\lambda \tau + \mu) = \mathbf{e}\Big( -m\lambda^2 \tau - 2m \lambda z - m (\lambda \mu + t) \Big) \sigma_{\beta}^*(\zeta) E_{k,m,\beta}^*(\tau,z,s)$$ for any $M = \begin{pmatrix} a & b \\ c & d \end{pmatrix} \in \tilde \Gamma$ and $\zeta = (\lambda,\mu,t) \in \mathcal{H}.$ 
\end{rem}

Using the argument of \cite{W1}, we see that $E_{k,m,\beta}^*(\tau,z,s)$ has the Fourier expansion $$E_{k,m,\beta}^*(\tau,z,s) = \sum_{\lambda \in \mathbb{Z}} q^{m \lambda^2} \zeta^{2m \lambda} \mathfrak{e}_{\lambda \beta} + \sum_{\gamma \in \Lambda'/\Lambda} \sum_{r \in \mathbb{Z} - \langle \gamma, \beta \rangle} \sum_{n \in \mathbb{Z} - q(\gamma)} c'(n,r,\gamma,s,y) q^n \zeta^r \mathfrak{e}_{\gamma},$$ where $q = \mathbf{e}(\tau)$ and $\zeta = \mathbf{e}(z)$ and the coefficient $c'(n,r,\gamma,s,y)$ represents the contribution from all $M \in \tilde \Gamma_{\infty} \backslash \tilde \Gamma$ other than the identity, given by 
$$c'(n,r,\gamma,s,y) = \frac{\sqrt{i}^{b^- - b^+ - 1}}{\sqrt{2m |\Lambda'/\Lambda|}} I(k-1/2,y,n-r^2/4m,s) \tilde L(n,r,\gamma,k+e/2+2s).$$ Here, $I(k,y,\omega,s)$ denotes the integral $$I(k,y,\omega,s) = y^{1-k-2s} e^{2\pi \omega y} \int_{-\infty}^{\infty} (t+i)^{-k} (t^2 + 1)^{-s} \mathbf{e}(-\omega y t) \, \mathrm{d}t,$$ and $\tilde L$ is the $L$-series $$\tilde L(n,r,\gamma,s) = \zeta(s - e - 1)^{-1} L(n,r,\gamma,s-1),$$ where $$L(n,r,\gamma,s) = \prod_{p \, \mathrm{prime}} \Big( \sum_{\nu=0}^{\infty} \mathbf{N}(p^{\nu}) p^{-\nu s} \Big),$$ and $\mathbf{N}(p^{\nu})$ is the number of zeros $(v,\lambda) \in \mathbb{Z}^{e+1} / p^{\nu} \mathbb{Z}^{e+1}$ of the polynomial $q(v + \lambda \beta - \gamma) + m \lambda^2 - r\lambda + n.$

\begin{rem} Gross and Zagier consider in \cite{GZ} the integral $$V_s(\omega) = \int_{-\infty}^{\infty} (t+i)^{-k} (t^2 + 1)^{-s} \mathbf{e}(-\omega t) \, \mathrm{d}t,$$ (notice that $k$ in that paper represents $\frac{k+1}{2}$ here), and they show that for $\omega \ne 0$, the completed integral $$V_s^*(\omega) = (\pi |\omega|)^{-s-k} \Gamma(s+k) V_s(\omega)$$ is an entire function of $s$ that satisfies the functional equation $$V_s^*(\omega) = \mathrm{sgn}(\omega) V_{1-k-s}^*(\omega).$$ Since $$I(k,y,\omega,s) = \frac{y^{1-s} e^{2\pi \omega y} (\pi |\omega|)^{s+k}}{\Gamma(s+k)} V_s^*(\omega y),$$ this extends $I(k,y,\omega,s)$ meromorphically to all $s \in \mathbb{C}$ and gives the functional equation $$I(k,y,\omega,s) = \mathrm{sgn}(\omega) (\pi |\omega| y^{-1})^{2s+k-1} \frac{\Gamma(1-s)}{\Gamma(s+k)} I(k,y,\omega,1-k-s), \; \; \omega \ne 0.$$ The integral for $\omega = 0$ is $$I(k,y,0,s) = 2\pi (-i)^k (2y)^{1-k-2s} \frac{\Gamma(2s+k-1)}{\Gamma(s)\Gamma(s+k)}.$$
\end{rem}

\begin{rem} The local $L$-series $$L_p(n,r,\gamma,s) = \sum_{\nu=0}^{\infty} \mathbf{N}(p^{\nu}) p^{-\nu s}$$ that occur in $L(n,r,\gamma,s)$ can be evaluated in the same way as the local $L$-series of \cite{W2}. Namely, for fixed $\gamma,\beta \in \Lambda'/\Lambda$ and $n \in \mathbb{Z}-q(\gamma)$, $m \in \mathbb{Z}-q(\beta)$, $r \in \mathbb{Z} - \langle \gamma, \beta \rangle$, we define discriminants $$\mathcal{D}' = d_{\beta}^2 d_{\gamma}^2 (-1)^{e/2 + 1} (4mn - r^2) |\Lambda'/\Lambda|$$ if $e$ is even and $$D' = 2m d_{\beta}^2 (-1)^{(e+1)/2} |\Lambda'/\Lambda|$$ if $e$ is odd. 

Define the  ``bad primes" to be $p=2$ as well as all odd primes dividing $|\Lambda'/\Lambda|$ or $m d_{\beta}^2$ or the numerator or denominator of $(n - r^2 / 4m) d_{\beta}^2 d_{\gamma}^2$, and set 

$$\mathcal{D} = \mathcal{D}' \cdot \prod_{\mathrm{bad}\, p} p^2, \; \; D = D' \cdot \prod_{\mathrm{bad}\, p} p^2.$$ If $e$ is even, then for primes $p \nmid \mathcal{D}$, $$L_p(n,r,\gamma,s) = \sum_{\nu=0}^{\infty} \mathbf{N}(p^{\nu}) p^{-\nu s} = \begin{cases} \frac{1}{1 - p^{e-s}} \Big[ 1 + \left( \frac{\mathcal{D}}{p} \right) p^{e/2 - s} \Big] : & r^2 / 4m - n \ne 0; \\ \\ \frac{1 - p^{e - 2s}}{(1 - p^{e-s})(1 - p^{1+e-2s})}: & r^2 / 4m - n = 0; \end{cases}$$ and if $e$ is odd, then for primes $p \nmid D$, $$L_p(n,r,\gamma,s) = \sum_{\nu=0}^{\infty} \mathbf{N}(p^{\nu}) p^{-\nu s} = \begin{cases} \frac{1}{1 - p^{e-s}} \Big[ 1 - \left( \frac{D}{p} \right) p^{(e-1)/2 - s}: & r^2 / 4m - n \ne 0; \\ \\ \frac{1 - \left(\frac{D}{p} \right) p^{(e-1)/2 - s}}{(1 - p^{e-s}) \Big[ 1 - \left( \frac{D}{p} \right) p{(e+1)/2 - s} \Big]}: & r^2 / 4m - n = 0; \end{cases}$$ where $\left( \frac{D}{p} \right), \left( \frac{\mathcal{D}}{p}\right)$ denote the Legendre (quadratic reciprocity) symbol. This gives the meromorphic extensions $$\tilde L(n,r,\gamma,s) = \begin{cases} \frac{L(s-1-e/2, \chi_{\mathcal{D}})}{\zeta(2s-2-e)} \prod_{\mathrm{bad}\, p} \frac{1 - p^{e+1-s}}{1 - p^{e+2-2s}} L_p(n,r,\gamma,s-1) : & r^2 / 4m - n \ne 0; \\ \\ \frac{\zeta(2s-3-e)}{\zeta(2s-2-e)} \prod_{\mathrm{bad}\, p} \frac{(1 - p^{e+1-s})(1 - p^{e-3-2s})}{1 - p^{e-2-2s}} L_p(n,r,\gamma,s-1) : & r^2 / 4m - n = 0; \end{cases}$$ for even $e$, and $$\tilde L(n,r,\gamma,s) = \begin{cases} \frac{1}{L(s - (e+1)/2, \chi_D)} \prod_{\mathrm{bad}\, p} \Big[ (1 - p^{e+1-s}) L_p(n,r,\gamma,s-1) \Big] : & r^2 / 4m - n \ne 0; \\  \\ \frac{L(s-(e+3)/2, \chi_D)}{L(s-(e+1)/2, \chi_D)} \prod_{\mathrm{bad}\, p} \Big[ (1 - p^{e+1-s}) L_p(n,r,\gamma,s-1) \Big] : & r^2 / 4m - n = 0; \end{cases}$$ for odd $e$.
\end{rem}

Together, this gives the analytic continuation of the Fourier coefficients $c'(n,r,\gamma,s,y)$ of $E_{k,m,\beta}^*(\tau,z,s)$ to $s \in \mathbb{C}$ (possibly with poles). Therefore, the series $$E_{k,m,\beta}^*(\tau,z,s) = \sum_{\lambda \in \mathbb{Z}} q^{m \lambda^2} \zeta^{2m \lambda} \mathfrak{e}_{\lambda \beta} + \sum_{n,r,\gamma} c'(n,r,\gamma,s,y) q^n \zeta^r \mathfrak{e}_{\gamma}$$ has an analytic continuation: the functional equations for $I(k,y,\omega,s)$ and $L(s,\chi_D)$, $L(s,\chi_{\mathcal{D}})$ imply convergence away from $\mathrm{Re}[s] > 2 - k/2.$

\begin{rem} We denote by $E_{k,m,\beta}(\tau,z)$ the series that results by naively evaluating the coefficient formula of \cite{W1} at $k = 3/2$ or $k = 2$ (without the weight $5/2$ correction). In the derivation of this formula it was assumed that $I(k-1/2,y,n-r^2 / 4m,0) = 0$  for $n - r^2 / 4m \le 0$ and that $\tilde L(n,r,\gamma,s)$ is holomorphic at $s = 0$. These assumptions are not generally satisfied when $k \le 5/2$, and $E_{k,m,\beta}(\tau,z)$ generally fails to be a Jacobi form in those cases. (In particular, $E_{k,m,\beta}(\tau,0)$ generally fails to be a modular form.)
\end{rem}

\section{A Petersson scalar product}

Recall that the Petersson scalar product on $S_k(\rho^*)$ is defined by $$(f,g) = \int_{\tilde \Gamma \backslash \mathbb{H}} \langle f(\tau), g(\tau) \rangle y^{k-2} \, \mathrm{d}x \, \mathrm{d}y, \; \; f,g \in S_k(\rho^*).$$ This is well-defined because cusp forms $f(\tau)$ satisfy the ``trivial bound'' $\|f(\tau)\| \le C \cdot y^{-k/2}$ for some constant $C$ (this is clear on the standard fundamental domain by continuity, and $\|f(\tau)\| y^{k/2}$ is invariant under $\tilde \Gamma$), and because $\langle f(\tau), g(\tau) \rangle y^{k-2}\, \mathrm{d}x \, \mathrm{d}y$ is invariant under $\tilde \Gamma$. More generally, we can define $(f,g)$ for any functions $f,g$ that transform like modular forms of weight $k$ and for which the integral above makes sense. (This includes the case that $f,g \in M_k(\rho^*)$ and only one of $f,g$ is a cusp form.) \\

In many cases it is useful to apply the following ``unfolding argument'' to evaluate $\langle f,g \rangle$, which is well-known. If $g(\tau)$ can be written in the form $$g(\tau) = \sum_{M \in \tilde \Gamma_{\infty} \backslash \tilde \Gamma} u \Big|_{k,\rho^*} M$$ for some function $u(\tau)$ that decays sufficiently quickly as $y \rightarrow \infty$, then for any cusp form $f$, \begin{align*} (f,g) &= \int_{\tilde \Gamma \backslash \mathbb{H}} \sum_{M \in \tilde \Gamma_{\infty} \backslash \Gamma} \langle f, u|_{k,\rho^*} M \rangle y^{k-2} \, \mathrm{d}x \, \mathrm{d}y \\ &= \int_{-1/2}^{1/2} \int_0^{\infty} \langle f,u \rangle y^{k-2} \, \mathrm{d}y \, \mathrm{d}x. \end{align*} This is because there is a unique representative of every class $M \in \tilde \Gamma_{\infty} \backslash \Gamma$ that maps the strip $[-1/2,1/2] \times [0,\infty)$ to itself, ``unfolding" the fundamental domain of $\tilde \Gamma \backslash \mathbb{H}$ to the strip.

\begin{ex} Taking the Petersson scalar product with the real-analytic Eisenstein series $$y^s E_k^*(\tau,s) = \sum_{M \in \tilde \Gamma_{\infty} \backslash \tilde \Gamma} (y^s \mathfrak{e}_0) \Big|_{k,\rho^*} M$$ gives $$\langle f, y^s E_k^*(\tau,s) \rangle = \int_0^{\infty} \underbrace{ \int_{-1/2}^{1/2} \langle f(\tau), \mathfrak{e}_0 \rangle \, \mathrm{d}x}_{=0} y^{k+s-2} \, \mathrm{d}y = 0$$ for all cusp forms $f$ and sufficiently large $\mathrm{Re}[s]$ (and more generally by analytic continuation).
\end{ex}

The more important example will be $$g(\tau) = y^s (E_{k,m,\beta}^*(\tau,0,s) - E_k^*(\tau,s)) = \sum_{M \in \tilde \Gamma_{\infty} \backslash \tilde \Gamma} \Big( \sum_{\lambda \ne 0} y^s \mathbf{e}(m \lambda^2 \tau) \mathfrak{e}_{\lambda \beta} \Big) \Big|_{k,\rho^*} M.$$ For any cusp form $f(\tau) = \sum_{\gamma \in \Lambda'/\Lambda} \sum_{n \in \mathbb{Z} - q(\gamma)} c(n,\gamma) q^n \mathfrak{e}_{\gamma}$ and large enough $\mathrm{Re}[s]$, the unfolding argument gives \begin{align*} (f,g) &= \sum_{\lambda \ne 0} \int_{-1/2}^{1/2} \int_0^{\infty} \langle f(\tau), \mathbf{e}(m \lambda^2 \tau) \mathfrak{e}_{\lambda \beta} \rangle y^{k+s-2} \, \mathrm{d}y \, \mathrm{d}x \\ &= 2 \cdot \sum_{\lambda=1}^{\infty} c(m \lambda^2, \lambda \beta) \int_0^{\infty} e^{-4\pi m \lambda^2 y} y^{k+s-2} \, \mathrm{d}y \\ &= 2 \cdot \Gamma(k+s-1) \sum_{\lambda = 1}^{\infty} \frac{c(m \lambda^2, \lambda \beta)}{(4\pi m \lambda^2)^{k+s-1}}. \end{align*}

\begin{rem} Series of the form $\sum_{\lambda = 1}^{\infty} \frac{c(m\lambda^2, \lambda \beta)}{\lambda^s}$ are closely related to symmetric square $L$-functions (see for example \cite{Sh}) and have meromorphic continuations to the entire plane (for which one can reduce to the scalar case, since the components of $f$ are modular forms of higher level). For $k \ge 5/2$, a M\"obius inversion argument was used in \cite{W1} to show that if a cusp form $f(\tau) = \sum_{n,\gamma} c(n,\gamma) q^n \mathfrak{e}_{\gamma}$ satisfies $\sum_{\lambda = 1}^{\infty} \frac{c(m \lambda^2, \lambda \beta)}{(4\pi m \lambda^2)^{k-1}} = 0$ for all $\beta \in \Lambda'/\Lambda$ and $m \in \mathbb{Z} - q(\beta)$, $m > 0$, then $f = 0$ identically. For $k = 3/2$ or $k = 2$, this argument does not seem rigorous when applied to the analytic continuation of $\sum_{\lambda = 1}^{\infty} \frac{c(m \lambda^2, \lambda \beta)}{(4\pi m \lambda^2)^{k + s - 1}}$ at $s=0$ although in practice it seems to hold.
\end{rem}

In the following sections, we will construct modular forms $Q_{3/2,m,\beta}(\tau) \in M_{3/2}(\rho^*)$ resp. cusp forms $Q_{2,m,\beta} - E_2 \in M_2(\rho^*)$ with rational coefficients that satisfy $$(f,Q_{k,m,\beta}) = 2 \Gamma(k-1) \cdot \sum_{\lambda = 1}^{\infty} \frac{c(m \lambda^2, \lambda \beta)}{(4\pi m \lambda^2)^{k+s-1}} \Big|_{s=0}.$$ For the above reason, the proof of \cite{W1} that such forms contain $S_k(\rho^*)$ within their span is not rigorous when $k = 3/2$ or $k = 2$ although the author is not aware of any examples where this fails.

\section{Weight $3/2$}

In weight $k = 3/2$, the $L$-series term is $$\tilde L(n,r,\gamma,3/2 + e/2 + 2s) = \begin{cases} \frac{1}{L(2s,\chi_D)} \prod_{\mathrm{bad}\, p} (1 - p^{e/2 - 1 - 2s}) L_p(n,r,\gamma,1/2 + e/2 + 2s): & n - r^2 / 4m \ne 0; \\ \\ \frac{L(2s,\chi_D)}{L(1+2s,\chi_D)} \prod_{\mathrm{bad}\, p} (1 - p^{e/2 - 1 - 2s}) L_p(n,r,\gamma,1/2 + e/2 + 2s) : & n - r^2 / 4m = 0. \end{cases}$$

This is holomorphic in $s = 0$ because the Dirichlet $L$-series $L(s,\chi_D)$ never has a pole at $s=0$ or a zero at $s=1$ and because the local $L$-factors $L_p(n,r,\gamma,1/2 + e/2 + 2s)$ are rational functions of $s$ with finitely many poles, while the dimension $e$ can be made arbitrarily large without changing the underlying discriminant form (and therefore the value of $\tilde L$). Note that $L_p(n,r,\gamma,1/2 + e/2 + 2s)$ may have a simple pole at $0$ if $e = 2$, but this is canceled by the factor $1 - p^{e/2 - 1 - 2s}$; in this case, we will write $$(1 - p^{e/2 - 1}) L_p(n,r,\gamma, 1/2 + e/2) = \lim_{s \rightarrow 0} (1 - p^{e/2 - 1 - 2s}) L_p(n,r,\gamma, 1/2 + e/2 + 2s)$$ by abuse of notation. \\

The coefficient formula \cite{W1} still requires a correction because the zero-value $I(1,y,0,0) = -\pi i$ is nonzero. This is easiest to calculate as a Cauchy principal value: $$I(1,y,0,0) = \lim_{s \rightarrow 0} \int_{-\infty}^{\infty} (t+i)^{-1} (t^2 + 1)^{-s} \, \mathrm{d}t = PV\Big[ \int_{-\infty}^{\infty} (t+i)^{-1} \Big] = -\pi i.$$ The corrected series \begin{align*} E_{3/2,m,\beta}^*(\tau,z,0) &= E_{3/2}(\tau,z) - \pi i \frac{\sqrt{i}^{b^- - b^+ - 1}}{\sqrt{2m |\Lambda'/\Lambda|}} \sum_{r^2 = 4mn} \tilde L(n,r,\gamma,3/2 + e/2) \\ &= E_{3/2}(\tau,z) - \pi \frac{(-1)^{(1 + b^- - b^+) / 4}}{\sqrt{2m |\Lambda'/\Lambda|}} \sum_{r^2 = 4mn} \tilde L(n,r,\gamma,3/2 + e/2) \end{align*} is holomorphic in $\tau$ and therefore defines a Jacobi form. (On the other hand, the exponent $(n,r) = (0,0)$ occurs in the sum on the right and therefore $E_{3/2,m,\beta}^*(\tau,z,0)$ will generally not have constant term $1 \cdot \mathfrak{e}_0$ and may even vanish identically.) We define $Q_{3/2,m,\beta}(\tau) = E_{3/2,m,\beta}^*(\tau,0,0).$ \\

Note that changing variables $\lambda \mapsto \lambda + \mu$ in the local $L$-series $$L_p(n,r,\gamma,s) = \sum_{\nu=0}^{\infty} p^{-\nu s} \# \Big\{ \mathrm{zeros} \, (v,\lambda) \, \mathrm{of} \, q(v + \lambda \beta - \gamma) + m \lambda^2 - r \lambda + n \, \bmod \, p^{\nu} \Big\}$$ gives the identity $$\tilde L(n,r,\gamma,s) = \tilde L(n+r\mu + m \mu^2, r + 2m \mu, \gamma + \mu \beta)$$ for all $\mu \in \mathbb{Z}$, which may simplify computations.

\begin{ex} Let $\Lambda = \mathbb{Z}^3$ with quadratic form $q(x,y,z) = 2xz + y^2$; then $M_{3/2}(\rho^*)$ is one-dimensional, spanned by \begin{align*} Q_{3/2,1,0}(\tau) &= \Big( \frac{1}{2} + 3q + 6q^2 + 4q^3 + ... \Big) (\mathfrak{e}_{(0,0,0)} - \mathfrak{e}_{(0,0,1/2)} - \mathfrak{e}_{(1/2,0,0)}) \\ &+ \Big( 4q^{3/4} + 12q^{11/4} + ... \Big) (\mathfrak{e}_{(0,1/2,0)} - \mathfrak{e}_{(0,1/2,1/2)} - \mathfrak{e}_{(1/2,1/2,0)}) \\ &+ \Big( - 6q^{1/2} - 12q^{3/2} - 12q^{5/2} - ... \Big) \mathfrak{e}_{(1/2,0,1/2)} \\ &+ \Big( -3q^{1/4} - 12q^{3/4} - 15q^{9/4} - ... \Big) \mathfrak{e}_{(1/2,1/2,1/2)},\end{align*} with constant term $\frac{1}{2} \mathfrak{e}_{(0,0,0)} - \frac{1}{2} \mathfrak{e}_{(0,0,1/2)} - \frac{1}{2} \mathfrak{e}_{(1/2,0,0)}.$ Unlike the case of weight $k \ge 5/2$, there is no way to produce a modular form with constant term $1 \cdot \mathfrak{e}_0$. (Following \cite{Bo}, the theta series in $M_{1/2}(\rho)$ act as obstructions to producing modular forms in $M_{3/2}(\rho^*)$ with arbitrary constant term.)
\end{ex}

\section{Weight $2$}

In weight $k = 2$, the $L$-series term is $$\tilde L(n,r,\gamma,2+e/2+2s) = \begin{cases} \frac{L(2s+1,\chi_{\mathcal{D}})}{\zeta(4s+2)} \prod_{\mathrm{bad}\, p} \frac{1 - p^{e/2 - 1 - 2s}}{1 - p^{e/2 - 2s}} L_p(n,r,\gamma,1+e/2 + 2s) : & n - r^2 / 4m \ne 0; \\ \\ \frac{\zeta(4s+1)}{\zeta(4s+2)} \prod_{\mathrm{bad}\, p} \frac{(1 - p^{e/2 - 1 - 2s})(1 - p^{-1-4s})}{1 - p^{-2-4s}} L_p(n,r,\gamma,1 + e/2 + 2s): & n - r^2 / 4m = 0. \end{cases}$$ Here, $\mathcal{D}$ denotes the discriminant $$\mathcal{D} = (r^2 - 4mn) |\Lambda'/\Lambda| d_{\beta}^2 d_{\gamma}^2 \prod_{\mathrm{bad}\, p} p^2.$$ This $L$-series has a pole in $s=0$ when $n - r^2 / 4m = 0$ or when $\mathcal{D}$ is a square, and in this case the residue there is \begin{align*} &\quad \mathrm{Res} \Big( \tilde L(n,r,\gamma,2+e/2 + 2s), s=0 \Big) \\ &= \frac{3}{\pi^2} \Big[\prod_{\mathrm{bad}\, p} \frac{1 - p^{e/2 - 1}}{1 + p^{-1}} L_p(n,r,\gamma,1 + e/2)\Big] \times \begin{cases} 1 : & n - r^2 / 4m \ne 0; \\ 1/2 : & n - r^2 / 4m = 0. \end{cases} \end{align*} As before, if $L_p(n,r,\gamma,1 + e/2 + 2s)$ has a pole in $s = 0$, (which can only happen in dimension $e = 2$), then we abuse notation and write $$(1 - p^{e/2 - 1}) L_p(n,r,\gamma,1+e/2) = \lim_{s \rightarrow 0} (1 - p^{e/2 - 1 - 2s}) L_p(n,r,\gamma,1+e/2 + 2s).$$ The pole of $\tilde L$ cancels with the zero of $I(3/2,y,n-r^2 / 4m,s)$ at $s=0$, whose derivative there is $$\frac{d}{ds} \Big|_{s=0} I(3/2,y,n-r^2/4m,s) = -16\pi^2 (1+i) y^{-1/2} \beta(\pi |4n - r^2 / m| y),$$ where $\beta(x) = \frac{1}{16\pi } \int_1^{\infty} u^{-3/2} e^{-xu} \, \mathrm{d}u$ (exactly as in the weight $3/2$ Eisenstein series in \cite{W2}). \\

For clarity, we can write $$E_2^*(\tau,z,0) = E_2(\tau,z) + \frac{1}{\sqrt{y}} \sum_{\gamma \in \Lambda'/\Lambda} \sum_{n \in \mathbb{Z} - q(\gamma)} \sum_{r \in \mathbb{Z} - \langle \gamma, \beta \rangle} A(n,r,\gamma) \beta(\pi (r^2 / m - 4n) y) q^n \zeta^r \mathfrak{e}_{\gamma},$$ with constants $$A(n,r,\gamma) = \frac{48 (-1)^{(4 + b^+ - b^-)/4}}{\sqrt{m |\Lambda'/\Lambda|}} \prod_{\mathrm{bad}\, p} \frac{1 - p^{e/2 - 1}}{1 + p^{-1}} L_p(n,r,\gamma,1 + e/2) \times \begin{cases} 1: & r^2 \ne 4mn; \\ 1/2 : & r^2 = 4mn; \end{cases}$$ if $(r^2 - 4mn) |\Lambda'/\Lambda|$ is square and $A(n,r,\gamma) = 0$ otherwise. We see that in general $E_2^*(\tau,0,0)$ is far from being a holomorphic modular form. \\

Instead, we define a family of cusp forms $Q_{2,m,\beta}^*(\tau,s) \in S_2(\rho^*)$ by taking the orthogonal projection of $y^s(E^*_{2,m,\beta}(\tau,0,s) - y^s E_2^*(\tau,s))$ to $S_2(\rho^*)$ with respect to the Petersson scalar product. Explicitly, if $e_1,...,e_n$ are an orthonormal basis of weight-$2$ cusp forms then $$Q_{2,m,\beta}^*(\tau,s) = \sum_{j=1}^n \Big( y^s E_{2,m,\beta}^*(\tau,0,s) - y^s E_2^*(\tau,s), e_j(\tau) \Big) \cdot e_j(\tau).$$ From the definition it is clear that for large enough $\mathrm{Re}[s]$, $Q_{2,m,\beta}^*(\tau,s)$ is the cusp form satisfying $$(f, Q_{2,m,\beta}^*(\tau,s)) = (f, y^s E_{2,m,\beta}^*(\tau,0,s)) = 2 \cdot \Gamma(1+s) \sum_{\lambda = 1}^{\infty} \frac{c(m \lambda^2, \lambda \beta)}{(4\pi m \lambda^2)^{1+s}}$$ for any cusp form $f(\tau) = \sum_{n,\gamma} c(n,\gamma) q^n \mathfrak{e}_{\gamma}.$ \\

\begin{rem} For any $\beta \in \Lambda'/\Lambda$ and $m \in \mathbb{Z} - q(\beta)$, $m > 0$, the Poincar\'e series of weight $2$ is defined by $$P_{2,m,\beta}(\tau) = \sum_{M \in \tilde \Gamma_{\infty} \backslash \tilde \Gamma} \Big( \mathbf{e}(m \tau) \mathfrak{e}_{\beta} \Big) \Big|_{2,\rho^*} M = \frac{1}{2} \sum_{c,d} (c \tau + d)^{-2} \mathbf{e}\Big( m (M \cdot \tau) \Big) \rho^*(M)^{-1} \mathfrak{e}_{\beta},$$ where $c,d$ runs through all pairs of coprime integers and $M \in \tilde \Gamma$ is any element with bottom row $(c,d).$ This series does not converge absolutely, but as shown in  \cite{Me}, $$\lim_{s \rightarrow 0} \sum_{M \in \tilde \Gamma_{\infty} \backslash \tilde \Gamma} \Big( y^s \mathbf{e}(m \tau) \mathfrak{e}_{\beta} \Big) \Big|_{2,\rho^*} M$$ is holomorphic in $\tau$ and therefore $P_{2,m,\beta}(\tau)$ defines a cusp form. The unfolding argument characterizes $P_{2,m,\beta}$ by $$(f,P_{2,m,\beta}) = \frac{c(m,\beta)}{4\pi m} \; \; \mathrm{for} \; \mathrm{any} \; \mathrm{cusp} \; \mathrm{form} \; f(\tau) = \sum_{n,\gamma} c(n,\gamma) q^n \mathfrak{e}_{\gamma}$$ as usual.
\end{rem}

\begin{rem} Writing $Q_{2,m,\beta}^*(\tau,s) = \sum_{\gamma \in \Lambda'/\Lambda} \sum_{n \in \mathbb{Z} - q(\gamma)} b(n,\gamma,s) q^n \mathfrak{e}_{\gamma}$, the fact that $Q_{2,m,\beta}^*(\tau,s) - y^s E_{k,m,\beta}^*(\tau,0,s)$ is orthogonal to all Poincar\'e series implies that \begin{align*} \frac{b(n,\gamma,s)}{4\pi n} &= \Big( Q_{2,m,\beta}^*(\tau,s), P_{2,n,\gamma} \Big) \\ &= \Big( y^s E_{2,m,\beta}^*(\tau,0,s), P_{2,n,\gamma} \Big) \\ &= \int_0^{\infty} c(n,\gamma,y,s) e^{-4\pi n y} y^s \, \mathrm{d}y, \end{align*} where $c(n,\gamma,y,s)$ is the coefficient of $q^n \mathfrak{e}_{\gamma}$ in $E_{2,m,\beta}^*(\tau,0,s).$
\end{rem}

\begin{defn} The \textbf{Poincar\'e square series} of weight $2$ is the quasimodular form $$Q_{2,m,\beta}(\tau) = E_2(\tau) + Q_{2,m,\beta}^*(\tau,0).$$
\end{defn}

It follows from the above remarks that $Q_{2,m,\beta}(\tau)$ differs from the computation of \cite{W1} as follows: we can write $$Q_{2,m,\beta}(\tau) = E_2(\tau,0) + \sum_{\gamma \in \Lambda'/\Lambda} \sum_{\substack{n \in \mathbb{Z} - q(\gamma) \\ n > 0}} b(n,\gamma) q^n \mathfrak{e}_{\gamma},$$ with coefficients \begin{align*} \frac{b(n,\gamma)}{4\pi n} &= \sum_{r \in \mathbb{Z} - \langle \gamma, \beta \rangle} A(n,r,\gamma) \int_0^{\infty} e^{-4\pi n y} \beta(\pi (r^2 / m - 4n) y) y^{-1/2} \, \mathrm{d}y \\ &= \frac{1}{16\pi} \sum_r A(n,r,\gamma) \int_0^{\infty} \int_1^{\infty} u^{-3/2} y^{-1/2} e^{4\pi n y(u-1) - \pi r^2 y u / m} \, \mathrm{d}u \, \mathrm{d}y \\ &= \frac{1}{16\pi} \sum_r A(n,r,\gamma) \int_1^{\infty} u^{-3/2} \Big( (r^2 / m - 4n) u + 4n \Big)^{-1/2} \, \mathrm{d}u \\ &= \frac{1}{32\pi n \sqrt{m}} \sum_r A(n,r,\gamma) \Big( |r| - \sqrt{r^2 - 4mn} \Big), \end{align*} i.e. \begin{align*} b(n,\gamma) &= \frac{1}{8 \sqrt{m}} \sum_{r \in \mathbb{Z} - \langle \gamma, \beta \rangle} A(n,r,\gamma) \Big( |r| - \sqrt{r^2 - 4mn} \Big). \end{align*} 

When $|\Lambda'/\Lambda|$ is square, it turns out that for fixed $n$ and $\gamma$, the sum above is finite and can be calculated directly. Otherwise, this tends to be a truly infinite series, and we will need some preparation to prove that $b(n,\gamma)$ are rational and to evaluate them with a finite computation.

\section{A Pell-type equation}

The condition $$\mathcal{D} = d_{\beta}^2 d_{\gamma}^2 (r^2 - 4mn)|\Lambda'/\Lambda| \prod_{\mathrm{bad} \, p} p^2 = \square$$ is equivalent to requiring $(a,b) = d_{\gamma} d_{\beta} (\sqrt{|\Lambda'/\Lambda|(r^2 - 4mn)}, r)$ to occur as an integer solution of the Pell-type equation $$a^2 - |\Lambda'/\Lambda| b^2 = -4 |\Lambda'/\Lambda| (d_{\beta}^2 m) (d_{\gamma}^2 n)$$ satisfying the congruence $b \equiv d_{\beta} d_{\gamma} \langle \gamma, \beta \rangle \, \bmod \, d_{\gamma}d_{\beta} \mathbb{Z}.$ We will study such equations in general.

\begin{defn} A \textbf{Pell-type problem} is a problem of the form $$\mathrm{find} \, \mathrm{all} \, \mathrm{integer} \, \mathrm{solutions} \, (a,b) \, \mathrm{of} \, a^2 - Db^2 = -4CD$$ for some $C,D \in \mathbb{N}.$
\end{defn}

The behavior of solutions is quite different depending on whether or not $D$ is square. If $D$ is a square, then the equation can be factored as $$(a - \sqrt{D} b) (a + \sqrt{D} b) = a^2 - Db^2 = -4CD,$$ from which it follows that there are only finitely many solutions and all are bounded by $|a|, \sqrt{D}|b| \le CD + 1.$ \\

Assume from now on that $D$ is nonsquare. In this case, the solutions of the Pell-type problem are closely related to the solutions of the true Pell equation $$a^2 - Db^2 = 1.$$ It follows from Dirichlet's unit theorem that there are infinitely many solutions $(a,b)$ of the Pell equation and all have the form $$a + \sqrt{D} b = \pm \varepsilon_0^n, \; \; n \in \mathbb{Z},$$ where $\varepsilon_0 \in \mathbb{Z}[\sqrt{D}]$ is the \textbf{fundamental solution} $\varepsilon_0 = a + \sqrt{D} b$, which is the minimal solution satisfying $\varepsilon_0 > 1.$ The problem of determining $\varepsilon_0$ is well-studied; see for example \cite{L} for an overview.

\begin{lem} Assume that $D$ is squarefree and let $K = \mathbb{Q}(\sqrt{D})$ with ring of integers $\mathcal{O}_K$. Then the solutions $(a,b)$ of the Pell-type equation $a^2 - Db^2 = -4CD$ are in bijection with elements $\mu \in \mathcal{O}_K$ having norm $C$. 
\end{lem}
\begin{proof} Let $(a,b)$ be any solution of the Pell-type equation and define $\mu = \frac{a + \sqrt{D} b}{2 \sqrt{D}}.$ This is an algebraic integer because its trace $\mu + \overline{\mu} = b$ and norm $\mu \overline{\mu} = C$ are both integers. Conversely, given any algebraic integer $\mu \in \mathcal{O}_K$ of norm $C$, we can define $(a,b)$ by $a + \sqrt{D} b = 2 \sqrt{D} \mu$.
\end{proof}

\begin{lem} Assume that $D$ is squarefree. Then there are finitely many elements $\mu_1,...,\mu_n \in \mathcal{O}_K$, all satisfying $0 \le \mathrm{Tr}_{K/\mathbb{Q}}(\mu_i) \le 2 \sqrt{C \varepsilon_1}$, such that $$\Big\{ \mu \in \mathcal{O}_K: \; N_{K/\mathbb{Q}}(\mu) = C \Big\} = \bigcup_{i=1}^n \mu_i \cdot \mathcal{O}_K^{\times,1}.$$
\end{lem}
Here, $\varepsilon_1$ is the fundamental solution to $N_{K/\mathbb{Q}}(\varepsilon_1) = 1$. In other words $\varepsilon_1$ is either the fundamental unit or its square if the fundamental unit has norm $-1$. Also, $$\mathcal{O}_K^{\times,1} = \{\varepsilon \in \mathcal{O}_K^{\times}: \; N_{K/\mathbb{Q}}(\varepsilon) = 1\}.$$
\begin{proof} Suppose $\mu$ is any solution of $N_{K/\mathbb{Q}}(\mu) = C$, and choose $n \in \mathbb{Z}$ such that $$|\log(\varepsilon_1^n \mu) - \log(\sqrt{C})|$$ is minimal. Then it follows that $$|\log(\varepsilon_1^n \mu) - \log(\sqrt{C})| \le \frac{1}{2} \log(\varepsilon_1).$$ In particular, $\varepsilon_1^n \mu \le \sqrt{C \varepsilon_1}$ and $\varepsilon_1^{-n} \mu^{-1} \le \sqrt{\varepsilon_1 / C}$. It follows that $$\Big| \mathrm{Tr}_{K/\mathbb{Q}}(\varepsilon_1^n \mu) \Big| = \Big| \varepsilon_1^n + C \varepsilon_1^{-n} \mu^{-1} \Big| \le 2 \sqrt{C \varepsilon_1}.$$ By replacing $\mu$ by $-\mu$ we may assume that $\mathrm{Tr}_{K/\mathbb{Q}}(\varepsilon_1^n \mu) \ge 0.$ \\

In particular, $\mu$ lies in the same $\mathcal{O}_K^{\times,1}$-orbit as a root of one of finitely many polynomials $X^2 + \lambda X + C$ with $0 \le \lambda \le \lfloor 2 \sqrt{C \varepsilon_1} \rfloor$, which also shows that there are finitely many orbits.
\end{proof}

\begin{ex} Consider the Pell-type equation $a^2 - 33b^2 = -528$ with $D = 33$ and $C = 4$. There are three orbits of elements $\mu \in \mathcal{O}_K = \mathbb{Z}[(1 + \sqrt{33}) / 2]$ with norm $4$, represented by $$\mu = 2, \; \mu = \frac{7 \pm \sqrt{33}}{2},$$ having traces $4$ and $7$. The bound in this case is $2 \sqrt{C \varepsilon_1} \approx 28$. Note that elements $\mu$ that are conjugate by $\mathrm{Gal}(K/\mathbb{Q})$ result in the same solutions to the Pell equation.
\end{ex}

\begin{rem} Let $b_0,n \in \mathbb{N}$. Reducing modulo $n$ shows that the set of solutions $(a,b)$ to $$a^2 - Db^2 = -4CD, \; \; b \equiv b_0 \, \bmod \, n$$ is also in bijection via $(a,b) \mapsto \mu = \frac{a + \sqrt{D} b}{2 \sqrt{D}}$ to a union of finitely many orbits (possibly none): $$\bigcup_{i=1}^n \mu_i \cdot \langle \varepsilon_{\mu_i} \rangle,$$ where the ``congruent fundamental solution" $\varepsilon_{\mu_i}$ is the minimal power of the fundamental solution $\varepsilon_1$ such that $\mathrm{Tr}_{K/\mathbb{Q}}(\mu_i (1 - \varepsilon_{\mu_i})) \equiv 0 \, (\bmod\, n).$ \\

 When $D$ is not squarefree, we can pull out the largest square factor of $D$ to reduce the equation $$a^2 - Db^2 = -4CD$$ to a squarefree Pell-type equation with congruence condition.
\end{rem}

\begin{lem} Fix $\gamma \in \Lambda'/\Lambda$ and $n \in \mathbb{Z} - q(\gamma)$, $n > 0$. Then the value of $$A(n,r,\gamma) \times \begin{cases} 1 : & r^2 \ne 4mn; \\ 2: & r^2 = 4mn; \end{cases}$$ depends only on the orbit of $d_{\gamma} d_{\beta}\sqrt{|\Lambda'/\Lambda|} (r + \sqrt{r^2 - 4mn})$ as a solution of the Pell-type equation $$a^2 - Db^2 = -4CD, \; \; D = |\Lambda'/\Lambda|, \; C = d_{\beta}^2 d_{\gamma}^2 mn,$$ with congruence condition $b \equiv d_{\beta} d_{\gamma} \langle \gamma, \beta \rangle$ mod $d_{\gamma} d_{\beta} \mathbb{Z}.$
\end{lem}

\begin{proof} Assume first that $\beta = 0$ and abbreviate $D = |\Lambda'/\Lambda|$. Multiplying $r + \sqrt{r^2 - 4mn}$ by the congruent fundamental solution $\varepsilon = a + b \sqrt{D}$ replaces $r$ by $$r \frac{\varepsilon + \varepsilon^{-1}}{2} + \sqrt{r^2 - 4mn} \frac{\varepsilon - \varepsilon^{-1}}{2} = ar + b \sqrt{D (r^2 - 4mn)},$$ and $r^2 - 4mn$ by $$(r^2 - 4mn) + 2Db^2 (r^2 - 4mn) + 4mn Db^2 + 2abr \sqrt{D(r^2 - 4mn)},$$ which is congruent to $r^2 - 4mn$ modulo the largest modulus whose square divides $D$. \\

Since $\beta = 0$, it follows that $E_{2,m,\beta}^*(\tau,z,s)$ arises from a weight-$3/2$ real-analytic Maass form (here the Eisenstein series $E_{3/2}^*(\tau,s)$) for the quadratic form $\tilde q(v,\lambda) = q(v) + m \lambda^2$ through the theta decomposition; in other words, the coefficient of $q^n \zeta^r \mathfrak{e}_{\gamma}$ in $E_{2,m,\beta}^*(\tau,z,s)$ equals the coefficient of $q^{n - r^2 / 4m} \mathfrak{e}_{(\gamma, r/2m)}$ in $E_{3/2}^*(\tau,s).$ In particular, this equality also holds for the real-analytic parts. The coefficients $A(n,r,\gamma)$ in the real-analytic part of $E_{3/2}^*(\tau,0)$ occur (up to a constant factor) as the coefficients of its shadow, which is a modular form of weight $1/2$ for the quadratic form $-\tilde q$. Using Skoruppa's strengthening of the Serre-Stark basis theorem (\cite{Sk}, Satz 5.1; see also (3.5) of \cite{BEF}), it is known that for any Weil representation $\rho : \tilde \Gamma \rightarrow \mathrm{Aut} \, \mathbb{C}[\Lambda'/\Lambda]$, $M_{1/2}(\rho)$ is spanned by modular forms that are $\mathbb{C}[\Lambda'/\Lambda]$-linear combinations of the theta series $$\vartheta_{\ell,b} = \sum_{\substack{v \in \mathbb{Z} \\ v \equiv b \, (2\ell)}} \mathbf{e}\Big( \frac{v^2}{4 \ell} \tau \Big), \; \; b \in \mathbb{Z},$$ where $\ell$ runs through divisors of $4N$ for which $N/\ell$ is squarefree (where $N$ is the level of the discriminant form $\Lambda'/\Lambda$), in which the Fourier coefficient of $q^n$ (multiplied by $1/2$ if $n = 0$) depends only on whether $\ell n$ is square and if so on the remainder of $\sqrt{4 \ell n}$ modulo $2\ell$. The previous paragraph implies this congruence for $n - r^2 / 4m$ for all $r + \sqrt{r^2 - 4mn}$ in the same orbit. \\

For general $\beta$, we can embed the space of Jacobi forms for $\rho_{\beta}^*$ of index $m$ as ``old" Jacobi forms of index $m d_{\beta}^2$ for the trivial action of the Heisenberg group via the Hecke-type operator $$U_{\beta} \Phi(\tau,z) = \Phi(\tau,d_{\beta} z)$$ and apply the argument for $\beta = 0.$
\end{proof}

\begin{prop} The Poincar\'e square series $Q_{2,m,\beta}(\tau)$ has rational Fourier coefficients.
\end{prop}

\begin{proof} The expression for the coefficients of $E_2(\tau,0)$ in \cite{W1} consists of special values of Dirichlet $L$-functions and finitely many local $L$-series, and these remain rational in weight $k=2$. Therefore, we need to show that the correction terms $$b(n,\gamma) = \frac{1}{8 \sqrt{m}} \sum_{r \in \mathbb{Z} - \langle \gamma, \beta \rangle} A(n,r,\gamma) \Big( |r| - \sqrt{r^2 - 4mn} \Big)$$ are rational. \\

This is easy to see when $|\Lambda'/\Lambda|$ is square, since $b(n,\gamma)$ is a finite sum of rational numbers. Assume that $|\Lambda'/\Lambda|$ is not square. \\

Suppose first that $\beta = 0$. By lemma 17, we can write $$\sum_{r \in \mathbb{Z}} A(n,r,\gamma) \Big( |r| - \sqrt{r^2 - 4mn} \Big) = \sum_{i=1}^N A( n, r, \gamma) \sum_r \Big( |r| - \sqrt{r^2 - 4mn} \Big),$$ where for each $i$, the sum over $r$ is taken over solutions $$(a,b) = d_{\gamma} \Big( \sqrt{|\Lambda'/\Lambda|(r^2 - 4mn)}, r\Big)$$ of the Pell equation with congruence condition coming from the orbit of an element $\mu_i$ of norm $C$ and minimal trace as in  lemma 14.  These solutions are given by $$r + \sqrt{r^2 - 4mn} = \pm \frac{2 \sqrt{|\Lambda'/\Lambda|}}{d_{\gamma}} \mu_i \varepsilon_i^n,$$ which runs through the solutions $r$ twice if $\overline{\mu_i} / \mu_i \in \mathcal{O}_K$ and once otherwise. The minimality of $\mathrm{Tr}_{K/\mathbb{Q}}(\mu_i)$ implies that the terms in the series are $$|r| - \sqrt{r^2 - 4mn} \in \{2\mu, 2\mu \varepsilon^{-n}, 2\overline{\mu} \varepsilon^{-n}: \; n \in \mathbb{N}\},$$ and \begin{align*} \sum_r \Big( |r| - \sqrt{r^2 - 4mn} \Big) &= \Big( \frac{\mu}{1 - \varepsilon^{-1}} + \frac{\overline{\mu} \varepsilon^{-1}}{1 - \varepsilon^{-1}} \Big) \times \begin{cases} 1 : & \overline{\mu} / \mu \in \mathcal{O}_K; \\ 2: & \mathrm{otherwise} \end{cases} \\ &= \frac{1}{N_{K/\mathbb{Q}}(1 - \varepsilon)} \Big( \mu - \overline{\mu} + \overline{\mu \varepsilon} - \mu \varepsilon \Big) \times \begin{cases} 1 : & \overline{\mu} / \mu \in \mathcal{O}_K; \\ 2: & \mathrm{otherwise}; \end{cases}  \end{align*} and we see that $\frac{1}{\sqrt{|\Lambda'/\Lambda|}} \sum_r \Big(|r| - \sqrt{r^2 - 4mn} \Big)$ is rational. Since $$A(n,r,\gamma) = \frac{1}{\sqrt{m |\Lambda'/\Lambda|}} \cdot \Big( \text{rational number} \Big),$$ we see that $b(n,\gamma)$ is rational. \\

The argument for general $\beta$ is essentially the same but slightly messier because $r + \sqrt{r^2 - 4mn}$ and $-r + \sqrt{r^2 - 4mn}$ generally occur as solutions of the Pell equation with different congruence conditions. In this case we can use the identity $b(n,\gamma) = b(n,-\gamma) = \frac{b(n,\gamma) + b(n,-\gamma)}{2}$ and consider both congruence conditions at once.
\end{proof}

The formula above has been implemented in SAGE and is available on the author's institutional webpage.

\section{Example: the class number relation}

In the simplest case of a unimodular lattice $\Lambda$ and index $m = 1$, the fact that $$Q_{2,1,0}(\tau) = E_2(\tau) = 1 - 24 \sum_{n=1}^{\infty} \sigma_1(n)q^n = 1 - 24 q - 72q^2 - 96q^3 - ...$$ (since the difference $Q_{2,1,0} - E_2$ is a scalar-valued cusp forms of weight $2$ and level $1$ so it vanishes) implies the Kronecker-Hurwitz class number relations. We explain this here. \\

The real-analytic Jacobi form $E_{2,1,0}^*(\tau,z,0)$ arises from the real-analytic correction of Zagier's Eisenstein series (in the form of example 15 of \cite{W2}), $$E_{3/2}^*(\tau,0) = 1 - 12 \sum_{n=1}^{\infty} H(n) q^{n/4} \mathfrak{e}_{n/2} - \frac{24}{\sqrt{y}} \sum_{n=-\infty}^{\infty} \beta(\pi y n^2) q^{-n^2 / 4} \mathfrak{e}_{n/2}$$ through the theta decomposition, where $H(n)$ is the Hurwitz class number of $n$. Therefore, $$E_{2,1,0}^*(\tau,z,0) = 1 -12 \sum_{n=1}^{\infty} \sum_{r=-\infty}^{\infty} H(4n - r^2) q^n \zeta^r + \frac{1}{\sqrt{y}} \sum_{n=-\infty}^{\infty} \sum_{r^2 - 4n = \square}A(n,r) \beta(\pi y (r^2 - 4n)) q^n \zeta^r$$ with constants $$A(n,r) = \begin{cases} -24: & r^2 - 4n = 0; \\ -48: & r^2 - 4n \ne 0. \end{cases}$$

It follows that \begin{align*} Q_{2,1,0}(\tau) &= 1 - 12 \sum_{n=1}^{\infty} \sum_{r=-\infty}^{\infty} H(4n - r^2) q^n + \frac{1}{8} \sum_{n=1}^{\infty} \sum_{r^2 - 4n = \square} A(n,r) \Big( |r| - \sqrt{r^2 - 4n} \Big) q^n \\ &= 1 - 12 \sum_{n=1}^{\infty} \sum_{r=-\infty}^{\infty} H(4n- r^2) q^n - 6 \sum_{n=1}^{\infty} \sum_{r^2 - 4n = \square} \Big( |r| - \sqrt{r^2 - 4n} \Big) q^n + 12 \sum_{n=1}^{\infty} n q^{n^2}. \end{align*} The identity $Q_{2,1,0} = E_2$ implies that for all $n \in \mathbb{N}$, $$\sum_{r=-\infty}^{\infty} H(4n-r^2) = 2 \sigma_1(n) - \frac{1}{2} \sum_{\substack{r \in \mathbb{Z} \\ r^2 - 4n = \square}} \Big( |r| - \sqrt{r^2 - 4n} \Big) + \begin{cases} \sqrt{n} : & n = \square; \\ 0: & \mathrm{otherwise}. \end{cases}$$ Here, $\frac{1}{2} \Big(|r| - \sqrt{r^2 - 4n}\Big)$ takes exactly the values $\min(d,n/d)$ as $d$ runs through divisors of $n$ (but counts $\sqrt{n}$ twice if $n$ is square); so this can be rearranged to $$\sum_{r = -\infty}^{\infty} H(4n-r^2) = 2 \sigma_1(n) - \sum_{d | n} \mathrm{min}(d, n/d).$$

\begin{rem} Mertens \cite{M} has given other proofs of this and similar class number relations using mock modular forms. It seems likely that we can recover other class number relations (possibly some of the other relations of \cite{M}) by studying the higher development coefficients (as defined in chapter 3 of \cite{EZ}) of the real-analytic Jacobi Eisenstein series $E_{2,1,0}^*(\tau,z,s)$ in the same way that we have studied its zeroth development coefficient $E_{2,1,0}^*(\tau,0,s)$ in this note, but we will not pursue that here.
\end{rem}

\section{Example: overpartition rank differences}

Consider the lattice $\Lambda = \mathbb{Z}^2$ with quadratic form $q(x,y) = x^2 - y^2$. There are no modular forms of weight $2$ for the dual Weil representation, and the $\mathfrak{e}_{(0,0)}$-component of the quasimodular Eisenstein series is $$E_2(\tau)_{(0,0)} = 1 - 16q - 24q^2 - 64q^3 - 72q^4 - 96q^5 - 96q^6 - 128q^7 - ...$$ This is a quasimodular form of level $4$ and we can verify by computing a few coefficients that it is

$$E_2(\tau)_{(0,0)} = E_2(2\tau) - 16 \sum_{n \, \mathrm{odd}} \sigma_1(n) q^n = 1 - 16 \sum_{n \, \mathrm{odd}} \sigma_1(n) q^n - 24 \sum_{n \, \mathrm{even}} \sigma_1(n/2) q^n.$$

The real-analytic Jacobi Eisenstein series of index $(m,\beta) = (1,0)$ corresponds to the real-analytic Eisenstein series for the lattice $\tilde \Lambda = \mathbb{Z}^3$ with quadratic form $q'(x,y,z) = x^2 - y^2 + z^2$ under the theta decomposition. It was shown in example 19 of \cite{W2} that the component of $\mathfrak{e}_{(0,0,0)}$ in the corresponding mock Eisenstein series is $$1 - 2q - 4q^2 - 8q^3 - ... = \sum_{n=0}^{\infty} (-1)^n \overline{\alpha(n)} q^n,$$ where $\overline{\alpha(n)}$ is the difference between the number of even-rank and odd-rank overpartitions of $n$. (We refer to \cite{BL} for the definition of overpartition rank differences and their appearance in weight-$3/2$ mock modular forms.) We will also need to understand the component of $\mathfrak{e}_{(0,0,1/2)}$ in this mock Eisenstein series. A quick computation shows that this is the series $$E_{3/2}(\tau)_{(0,0,1/2)} = -4q^{3/4} - 4q^{7/4} - 12q^{11/4} - 8q^{15/4} - 12q^{19/4} - 12q^{23/4} - 16q^{27/4} - ...$$

\begin{lem} The coefficient of $q^{n/4}$ in the series $E_{3/2}(\tau)_{(0,0,1/2)}$ is $$\begin{cases} -12 H(n) : & n \equiv 3 \, \bmod 8; \\ -4 H(n): & n \equiv 7 \, \bmod 8; \end{cases}$$ where $H(n)$ is the Hurwitz class number.
\end{lem}
We remark without proof that this series has an interesting closed form: \begin{align*} &\quad -4q^{3/4} - 4q^{7/4} - 12q^{11/4} - 8q^{15/4} - 12q^{19/4} - ... \\ &= -12 \sum_{n \equiv 3 \, (8)} H(n) q^{n/4} -4 \sum_{n \equiv 7\, (8)} H(n) q^{n/4} \\ &= -4q^{-1/4} \Big( \prod_{n=1}^{\infty} \frac{1 + q^n}{1-q^n} \Big) \Big( \frac{q}{1+q} - \frac{3q^4}{1+q^3} + \frac{5q^9}{1 + q^5} - \frac{7q^{16}}{1 + q^7} + \frac{9q^{25}}{1+q^9} - ... \Big). \end{align*}

\begin{proof} We can use the exact formula for the coefficients given by Bruinier and Kuss \cite{BK} in the form of \cite{W2} (where it was only stated for even-dimensional lattices): for an odd-dimensional lattice of dimension $e$, the coefficient $c(n,\gamma)$ of $E_k(\tau)$ is given by $$\frac{(2\pi)^k n^{k-1} (-1)^{b^+ / 2} L(k-1/2, \chi_{\mathcal{D}})}{\sqrt{|\Lambda'/\Lambda|} \Gamma(k) \zeta(2k-1)} \sum_{d | f} \mu(d) \chi_{\mathcal{D}}(d) d^{1/2-k} \sigma_{2 - 2k}(f/d) \prod_{p | ( 2|\Lambda'/\Lambda| )} \Big[ \frac{1 - p^{e/2 - k}}{1 - p^{1-2k}} L_p(n,\gamma,k+e/2-1) \Big].$$ For the lattice $\mathbb{Z}$ with quadratic form $q(x) = x^2$ (where $E_{3/2}$ is Zagier's mock Eisenstein series), and $\gamma = 1/2$ and $n \in \mathbb{Z} - q(\gamma)$, it is not hard to see that the local factor at $p=2$ is $$L_2(n,\gamma,s) = \begin{cases} 1 : & 4n \equiv 3 \, (\bmod 8); \\ (2^s + 1) / (2^s - 1): & 4n \equiv 7 \, (\bmod 8); \end{cases}$$ since $n$ always has valuation $-2$ modulo $p=2$, resulting in the values $(1 - 2^{-1}) L_2(n,\gamma,1) = 1/2$ if $4n \equiv 3 \, (8)$ resp. $(1-2^{-1}) L_2(n,\gamma,1) = 3/2$ if $4n \equiv 7 \, (8)$. On the other hand, for the lattice $\mathbb{Z}^3$ with quadratic form $q(x,y,z) = x^2 - y^2 + z^2$, the local factor is always $$L_2(n,\gamma,s) = \frac{2^s}{2^s - 4}$$ with $\lim_{s \rightarrow 0} (1 - 2^{-2s}) L_2(n,\gamma,2+2s) = 1$. Since all other terms in the formula are the same between the two lattices (other than an extra factor of $1/2$ from $\frac{1}{\sqrt{|\Lambda'/\Lambda|}}$), and the coefficient of $q^{n/4}$ in Zagier's mock Eisenstein series is $-12H(n)$, we get the claimed formula.
\end{proof}

In example 19 of \cite{W2} it was shown that the real-analytic correction of $E_{3/2}(\tau)$ for the lattice $\tilde \Lambda$ is $$E_{3/2}^*(\tau,0) = E_{3/2}(\tau) + \frac{1}{\sqrt{y}} \sum_{\gamma \in \tilde \Lambda'/\tilde \Lambda} \sum_{\substack{n \in \mathbb{Z} - q(\gamma) \\ n \le 0}} a(n,\gamma) \beta(-4\pi n y) q^n \mathfrak{e}_{\gamma}$$ with shadow \begin{align*} \sum_{\gamma,n} a(-n,\gamma) q^n \mathfrak{e}_{\gamma} &= -8\Big( 1 + 2q + 2q^4 + ... \Big) (2 \mathfrak{e}_{(0,0,0)} + \mathfrak{e}_{(1/2,1/2,0)} + \mathfrak{e}_{(0,1/2,1/2)}) - \\ &\quad \quad -8 \Big( 2q^{1/4} + 2q^{9/4} + 2q^{25/4} + ... \Big) (2 \mathfrak{e}_{(1/2,1/2,1/2)} + \mathfrak{e}_{(0,0,1/2)} + \mathfrak{e}_{(1/2,0,0)}). \end{align*}
Therefore the constants of section 5 for square $r^2 - 4n$ are $$A(n,r,0) = \begin{cases} -16: & r^2 - 4n\; \text{odd or zero}; \\ -32: & r^2 - 4n \; \text{even and nonzero}. \end{cases}$$ The formula for the Poincar\'e square series of index $(1,0)$ implies that the coefficient of $q^n \mathfrak{e}_{(0,0)}$ in $Q_{2,1,0}$ is $$\sum_{r = -\infty}^{\infty} -|\overline{\alpha}(n - r^2)| - 4 \sum_{r \, \mathrm{odd}} H(4n-r^2) + \frac{1}{8} \sum_{\substack{r \in \mathbb{Z} \\ r^2 - 4n = \square}} A(n,r,0) \Big( |r| - \sqrt{r^2 - 4n} \Big) + \begin{cases} 4: & n =\square; \\ 0: & \mathrm{otherwise}; \end{cases}$$ if $n$ is even and $$\sum_{r=-\infty}^{\infty} -|\overline{\alpha}(n - r^2)| - 12 \sum_{r \, \mathrm{odd}} H(4n-r^2) + \frac{1}{8} \sum_{\substack{r \in \mathbb{Z} \\ r^2 - 4n = \square}} A(n,r,0) \Big( |r| - \sqrt{r^2 - 4n} \Big) + \begin{cases} 4: & n = \square; \\ 0: & \mathrm{otherwise}; \end{cases} $$ if $n$ is odd. The additional $4$ at the end if $n$ is square is due to the constant term in the mock Eisenstein series $E_{3/2}$ being $1$ rather than $-1$: $$E_{3/2}(\tau)_{(0,0,0)} = 1 - \sum_{n=1}^{\infty} |\overline{\alpha}(n)| q^n = -\sum_{n=0}^{\infty} |\overline{\alpha}(n)| + 2,$$ and because we use the convention $\overline{\alpha}(0) = 1.$ As before, $\frac{|r| - \sqrt{r^2 - 4n}}{2}$ takes exactly the values $\mathrm{min}(d,n/d)$ for divisors $d$ of $n$ (but counts $\sqrt{n}$ twice if $n$ is square); and one can show that if $n$ is odd and $r^2 - 4n$ is square, then $r^2 - 4n$ is always even, while if $n$ is even, then $r^2 - 4n$ is even exactly when the divisor $d = \frac{|r| - \sqrt{r^2 - 4n}}{2}$ and $n/d$ are both even. \\

Denote $\lambda_1(n) = \frac{1}{2} \sum_{d | n} \min(d,n/d)$ as in \cite{M}. Comparing coefficients with the Eisenstein series $E_2(\tau)_{(0,0)}$ gives the following formula:

\begin{prop} If $n \in \mathbb{N}$ is odd, then $$\sum_{r=-\infty}^{\infty} |\overline{\alpha}(n - r^2)| = -16 \lambda_1(n) + 16 \sigma_1(n) - 12 \sum_{r \, \mathrm{odd}} H(4n-r^2) + \begin{cases} 4: & n = \square; \\ 0: & \mathrm{otherwise}. \end{cases} $$ If $n \in \mathbb{N}$ is even, then $$\sum_{r=-\infty}^{\infty} |\overline{\alpha}(n- r^2)| = -8 \lambda_1(n) - 16 \lambda_1(n/4) + 24 \sigma_1(n/2) - 4 \sum_{r \, \mathrm{odd}} H(4n - r^2) + \begin{cases} 4: & n = \square; \\ 0: & \mathrm{otherwise}. \end{cases}$$
\end{prop}
Here, we set $\lambda_1(n/4) = 0$ if $n$ is not divisible by $4$, and $\overline{\alpha}(n) = H(n) = 0$ for $n < 0$. Note that this can also be expressed as a relation among Hurwitz class numbers since $|\overline{\alpha}(n)|$ itself can be written in terms of Hurwitz class numbers, as observed in corollary 1.2 of \cite{BL2}.

\bibliographystyle{plainnat}
\bibliography{\jobname}

\end{document}